\documentclass[12pt,english]{article}

\usepackage[T1]{fontenc}
\usepackage[utf8]{inputenc}
\usepackage[english]{babel}

\usepackage{authblk}
\usepackage{amsthm,amsmath,amsfonts,amssymb}
\usepackage{latexsym}
\usepackage{enumitem}
\usepackage{subcaption}
\usepackage{tikz}
\usetikzlibrary{arrows,3d,patterns}

\def\centerarc[#1](#2)(#3:#4:#5)
    { \draw[#1] ($(#2)+({#5*cos(#3)},{#5*sin(#3)})$) arc (#3:#4:#5); }

\usepackage{a4wide}
\usepackage{times}
\usepackage{array}
\usepackage{hhline}
\usepackage{color}
\usepackage{todonotes}
\usepackage{changes}
\usepackage{mathtools}
\colorlet{Changes@Color}{orange}

\numberwithin{equation}{section}

\theoremstyle{plain}
\newtheorem{theorem}{Theorem}[section]
\newtheorem{proposition}{Proposition}[section]

\newtheorem{corollary}{Corollary}[section]
\theoremstyle{definition}
\newtheorem{definition}{Definition}[section]

\theoremstyle{remark}
\newtheorem{remark}{Remark}[section]

\def\la{\lambda}

\def\vfi{\varphi}

\def\d{\partial}

\def\bn{{\bf n}}

\def\ker{\hbox{ker}\,}
\def\image{\hbox{Im}\,}
\def\dim{\hbox{dim}\,}
\def\Lip{\hbox{Lip}\,}
\def\di{\hbox{div}\,}
\def\supp{\hbox{supp}\,}

\def\cD{\mathcal D}

\def\cH{\mathcal H}
\def\cN{\mathcal N}
\def\cM{\mathcal M}

\def\bN{\mathbb{N}}

\def\bR{\mathbb{R}}

\def\mvint_#1{\mathchoice
          {\mathop{\vrule width 6pt height 3 pt depth -2.5pt
                  \kern -9pt \intop}\nolimits_{\kern -3pt #1}}%
          {\mathop{\vrule width 5pt height 3 pt depth -2.6pt
                  \kern -6pt \intop}\nolimits_{#1}}%
          {\mathop{\vrule width 5pt height 3 pt depth -2.6pt
                  \kern -6pt \intop}\nolimits_{#1}}%
          {\mathop{\vrule width 5pt height 3 pt depth -2.6pt
                  \kern -6pt \intop}\nolimits_{#1}}}

\begin{document}

\title{A stationary heat conduction problem}

\author[{}]{Piotr Rybka{\footnote{P.Rybka@mimuw.edu.pl}}\ }
\author[{}]
{Anna Zatorska-Goldstein{\footnote{A.Zatorska-Goldstein@mimuw.edu.pl}}}

\affil[{}]
{{\small Institute of Applied
Mathematics and Mechanics\\ University of Warsaw\\ ul. Banacha 2,
02-097 Warsaw, Poland}}

\maketitle

\begin{abstract}
We study a basic linear elliptic equation on a lower dimensional rectifiable set $S$ in $\bR^N$  with the Neumann boundary data. Set $S$ is a support of a finite Borel measure $\mu$. We will use the measure theoretic tools to  interpret the equation and the Neumann boundary condition. For this purpose we recall the Sobolev-type space dependent on the measure $\mu$.  We establish existence and uniqueness of  weak solutions provided that an appropriate source term is given.
\end{abstract}

\smallskip



\section{Introduction}

We study here an old problem of determining the stationary heat distribution in a conductor $S$ in the ambient space $\Omega\subset \bR^N$, when the conductivity tensor $A$ and the heat sources $Q$ are given. We assume that the conductor $S$ is insulated at the boundary of $\Omega$.
The main objective of the paper is to investigate the case of "thin" $S$, i.e. the case when the set is a low dimensional structure. Very roughly, our problem may be written as,
\begin{equation}\label{r1}
\begin{array}{ll}
\di \left(A \nabla u \right)+Q=0&\hbox{in } S,\\
\left[ A \nabla u, \nu \right]|_{\d\Omega} = 0&\hbox{on } S\cap \d\Omega,
\end{array}
\end{equation}
where $\nu$ is the outward normal unit vector of $\partial \Omega$.
Although the system of equations looks familiar the problem becomes non-standard, when $S$ is a low dimensional rectifiable set.
To give proper meaning to these equations, we follow the approach proposed by Bouchitt\'e, Buttazzo and Seppecher in \cite{energies} and we consider a measure $\mu$, supported in $S$, which is singular with respect to the Lebesgue measure in $\bR^N$. We assume further  that $Q$  and $A$ are, respectively, a scalar and a tensor valued measure; both measures are absolutely continuous with respect to $\mu$.
Such measure-oriented point of view  proved to be very fruitful when dealing with variational problems considered on low dimensional structures in $\bR^N$, for an example see the anisotropic  shape optimization problem considered by Bouchitt\'e and Buttazzo in \cite{BB-JEMS}. We refer the reader to \cite{fragala} for an introduction to the theory.  
The novelty of our paper is to consider the Neumann boundary condition \eqref{r1}${}_2$; this will  be explained later in more detail.

We employ Sobolev-type space $H^{1,p}_\mu$ introduced by Bouchitt\'e and Buttazzo, \cite{energies}. Its definition and basic properties are discussed in Section \ref{sobolev spaces}. In particular, we explain there the notion of the tangential gradient operator $\nabla_\mu $ for functions from $H^{1,p}_\mu$ and the notion of the tangent space of $\mu$ at $x$, denoted by $T_\mu(x)$.

A natural way to handle the question of existence of solutions to \eqref{r1} is to use a variational approach. One may expect \eqref{r1} to be the Euler-Lagrange equation of the natural energy functional
\begin{equation}\label{dE}
 E(u) = \frac12 \int_\Omega (A(x)\nabla u, \nabla u)\,d \mu - \langle Q, u\rangle, \qquad u \in C^\infty_c(\bR^N).
\end{equation}
Energy $E$ is well-defined for smooth functions, but this space is not suitable from the point of view of the calculus of variations. As in \cite{energies}, the relaxation of $E$, denoted by $E_\mu$, may be calculated explicitely, see (Section \ref{elliptic problem}, Proposition \ref{pr3.1}). It is defined as
$$
E_\mu(u) = \frac12 \int_\Omega (A_\mu(x)\nabla_\mu u, \nabla_\mu u)\,d \mu - \langle Q, u\rangle, \qquad u \in
H^{1,2}_\mu,
$$
where $A_\mu$ is given by \eqref{rA}.

\medskip

As for the measure $\mu$, we assume it to be a \textit{multijunction measure} in the sense considered by Bouchitt\'e and Fragal\`a in \cite{BF-JFA}. However, we point out that, contrary to the approach in the papers \cite{BF-SIAM} and \cite{BF-JFA}, we do not assume a Poincar\'e-type inequality to hold on $S$, allowing in particular occurence of multidimensional junctions. The issue of the Poincar\'e inequality is addressed further in Section \ref{sobolev spaces} and examples are given in Section \ref{examples}.

{Our precise assumptions, denoted further by \textbf{[S]}, are as follows:
\begin{itemize} \label{[S]}
\item $\Omega$ is a smooth bounded domain in $\bR^N$;
\item a compact set $S$ contained in $\overline{\Omega}$ is of the form
$$
\mathrm{supp}\, \mu = S = \bigcup_{j=1}^J S_j,
$$
where each $S_j$ is a compact manifold of class $\mathcal{C}^2$ of dimension $k_j<N$;
\item the measure $\mu$ is of the form
$$
\mu=\sum_{j=1}^J \mu_j, \qquad \text{where} \quad \mu_j := \cH^{k_j} \llcorner S_j, \qquad j=1, \ldots J;
$$
\item these measures are mutually singular, i.e.
\begin{equation*}
\mu_j(S_i) = 0 \qquad \text{for all $j \neq i$};
\end{equation*}
\item for each $i=1,\ldots,J$, we assume that the boundary of the manifold $S_i$ is contained in
$\d \Omega \cup \bigcup_{j\neq i} \bar S_j$;
\item the relative interior of  $S_i$ does not intersect the boundary $\d \Omega$.
\end{itemize}
The ambient space $\Omega$ is, in fact, a design region restricting the position of set $S$. Remark \ref{rem41} at the end of Section \ref{bd condition} addresses the lack of influence of $\Omega$ on the solutions to (\ref{r1}).}
\medskip

We employ the following definition:
\begin{definition}
Let us assume that the conditions \textbf{[S]} on measure $\mu$, and (\ref{z2}) on measure $Q$ hold.
We say that a function $u\in H^{1,2}_\mu
$ is a {\it distributional solution} to the Neumann boundary problem \eqref{r1} if
\begin{equation}\label{rel-r1}
 \int_\Omega (A_\mu \nabla_\mu u, \nabla_\mu \vfi)\,d\mu - \langle Q, \vfi \rangle =0\qquad\hbox{for all }
 \vfi \in C^\infty_c(\bR^N).
\end{equation}
\end{definition}

We stress that the meaning of the boundary condition (\ref{r1}${}_2$) requires  clarification. First of all, we should consider $A_\mu \nabla_\mu u$, not $A \nabla_\mu u$. Secondly, the way we defined the weak solutions implies that $[A_\mu \nabla_\mu u, \nu]$ is  the zero distrubution over $\partial\Omega$, i.e.  $[A_\mu \nabla_\mu u, \nu]\in \mathcal{D}(\partial\Omega)$. We will show that this  a nicer object, but it cannot be considered in a pointwise manner.

To put it differently,  $A_\mu \nabla_\mu u$ is at best an element of $L^2(\Omega, \mu)$.
However, when $u$ is a minimizer of $E_\mu$,  we can interpret $\di A_\mu \nabla_\mu u$ as a measure, which permits us to  use the theory developed by Chen and Frid, \cite{frid}. In this way
we are able to interpret $\left[ A_\mu \nabla_\mu u, \nu \right]$ as a normal trace of a measure with bounded divergence. The issue of boundary condition is discussed in Section \ref{bd condition}.

As for the source term $Q$, in order to ensure solvability of (\ref{r1}), it is clear that $Q$ must be perpendicular to the kernel of $\di A_\mu \nabla_\mu $. The source term $Q$ must be also somehow subordinate to $\mu$. The first guess might be that $Q\in (H^{1,2}_\mu
)^*$. However, as explained in Section \ref{bd condition}, the interpretation of the boundary condition requires $Q$ that be a measure of a finite total variation.

\medskip

Our main result is then stated in the following Theorem:
\begin{theorem}\label{tw0}
Let us assume that the conditions \textbf{[S]} on measure $\mu$, and (\ref{z2}) on measure $Q$ hold. In addition, the function $A:S\to M(N\times N)$ takes values in symmetric, non-negative matrices satisfying (\ref{z3}). Moreover, $\langle Q, h \rangle =0$ for all $h$ in the
kernel of $\nabla_\mu$. Then there exists a unique distributional solution to \eqref{r1} which is perpendicular (in the $L^2$ scalar product) to $\mathrm{ker}\, \nabla_\mu$.
Moreover,  the solution satisfies the boundary condition in the following sense: $\left[ A_\mu \nabla_\mu u, \nu \right]=0 $ as a continuous linear functional on $C^2(\partial \Omega)$.
\end{theorem}
This Theorem follows from Theorem \ref{tw1} in Section \ref{elliptic problem} and Theorem \ref{tw b} in Section \ref{bd condition}. Uniqueness is addressed in Corollary \ref{wn-jed}. 
Examples of explicit boundary problems are discussed in Section \ref{examples}.

\section{Sobolev spaces $H^{1,p}_\mu$} \label{sobolev spaces}

\subsection{Definitions}

In the problem we study, $\Omega$ is the ambient space, containing $S$, which is  the support of a measure $\mu$. Set $S$ interpreted as the heat conductor. We study an elliptic problem in $S$, however, we are interested in the Neumann-type boundary conditions on $S\cap \d\Omega$.

We need properly defined Sobolev spaces to study weak solvability of differential equations like (\ref{r1}). We require a definition, which is general, yet easy to use. It seems that the Sobolev space $H^{1,p}_\mu$ introduced for any $p\in[1,\infty)$ by Bouchitte-Buttazzo-Seppecher in \cite{energies} is the right tool. We will briefly recall this definition. A recent discussion 
on other possible definitions of Sobolev space on $S$ is in \cite{louet}.

We begin with the notion of a tangent space to a measure, there is a number of ways to introduce it. We use the approach presented in \cite{energies} and further exposed in \cite{fragala}, \cite{mandallena} in a more reader friendly way.

We assume that $\mu$ is a positive, Radon measure on $\bR^N$. If we set
$$
\cN_\mu =\{v \in C^\infty_c(\bR^N):\ v = 0\hbox{ in }\supp \mu\},
$$ then we introduce
$$
\mathfrak{N}_\mu:= \{ w\in C^\infty_c(\bR^N; \bR^N):\ \exists v\in\cN_\mu, \ w =\nabla v \hbox{ in }\supp \mu\}.
$$
We define the multifunction $N_\mu: \bR^N \rightrightarrows\bR^N $ by
$$
N_\mu (x):= \{ w(x)\in \bR^N:\ w\in \mathfrak{N}_\mu\}.
$$
The vector subspace of $\bR^N$ is called  the {\it normal space to $\mu$ at $x$}. The vector space
$T_\mu(x)\subset \bR^N$ is defined as the orthogonal complement of $N_\mu (x)$ and it is called the  {\it tangent space to $\mu$ at $x$}.


We introduce
$\Pi_\mu(x,\cdot)$ as 
the orthogonal projection of $\bR^N$ onto $T_\mu(x)$. 
Due to measurability of $T_\mu(\cdot)$, see \cite{energies}, the projection is $\mu$-measurable.

For any $u\in \cD(\bR^n)$, we define,  for $\mu$-a.e.  $x\in \bR^n$, the tangential gradient
$$
\nabla _\mu u(x) = \Pi_\mu(x, \nabla u(x)).
$$
The space $H^{1,p}_\mu
$ is defined as a completion of $\cD(\bR^n)$ in the following norm
$$
\| u \|_{1,p,\mu} = \left(\| u \|^2_{L^p_\mu(\Omega)} + \|\nabla_\mu u \|^2_{L^p_\mu(\Omega)} \right)^{1/2}.
$$
It turns out to be a reflexive Banach space for $p \in (1, \infty)$.
We should stress that we could define the Sobolev space by weak derivatives, this is done in \cite{BBconverge}. This space is denoted by $W^{1,p}_\mu
$, in general $H^{1,p}_\mu
$ is closed subspace of  $W^{1,p}_\mu
$, see \cite{BBconverge} for details.

The weak convergence on space $H^{1,p}_\mu
$
is introduced in a natural way,
$$
u_k \rightharpoonup u\hbox{ weakly in }H^{1,p}_\mu
\Leftrightarrow
\left\{
\begin{array}{l}
 u_k\rightharpoonup u\hbox{ weakly in }L^p_\mu(\Omega),\\
 \nabla_\mu u_k\rightharpoonup \nabla_\mu u\hbox{ weakly in }L^p_\mu(\Omega).
\end{array}
\right.
$$

\medskip

In a typical situation we deal with, described in the Introduction, i.e. $\mu$ being supported on a finite union of smooth compact manifolds and agreeing local with the natural Hausdorff measure, the definitions of the tangent space and the Sobolev space $H^{1,p}_\mu
$ are intuitive.

\begin{proposition}(see \cite[Corollary 5.4]{energies})\label{pr-roz}
Let $S$ be a $C^2$ compact manifold in $\bR^N$ of dimension $k\le N$, let $T_S(x)$ be the tangent space at every $x\in S$, let $\mu = \cH^k \llcorner S$. Then, for every $p\in [1,+\infty)$ we have
$$
 T_\mu(x) = T_S(x)\qquad\mu-a.e.
$$
\end{proposition}

\begin{proposition}(see \cite[Lemma 2.2]{BF-JFA})\label{pr-sob}
Assume that $S$ is a finite union of $C^2$ compact manifolds, $S = \bigcup_{j=1}^J$, dim$\,S_j =k_j\le N$ and the measures  $\mu_j = \cH^{k_j}\llcorner S_j$ are mutually singular. We set $\mu = \sum_{j=1}^J \mu_j$ and we denote by $T_{S_j}(x)$ the tangent space to $S_j$ at $x$. Then, for every  $p\in [1,+\infty)$ we have
$$
 T_\mu(x) = T_{S_j}(x)\qquad\mu_j-a.e.
$$
If $S=\mathrm{supp}\, \mu \subset \overline{\Omega}$  and $\Omega$ is an open domain in $\bR^N$ and $\mu( S \cap \partial \Omega ) = 0$,  then
$$
H^{1,p}_\mu
=  \{ u \in L^p(S,\mu):\ u|_{S_j}\in H^{1,p}(S_j),\ j=1, \ldots,  J\}.
$$
\end{proposition}

\subsection{On validity of the Poincar\'e inequality} \label{Poincare}

The important tool in the analysis of the well posedness of the problem \eqref{r1} is the Poincar\'e inequality. In case $S$ is a smooth, compact manifold of dimension $k < N$, the global Poincar\'e inequality holds, i.e. there is $C_P>0$ such that
\begin{equation}\label{point-g}
\| u -  u_S\|_{L^2(\Omega,\mu)}^2 \le C_p  \| \nabla_\mu u\|^2_{L^2(\Omega,\mu)}\qquad u\in H^{1,2}_\mu
\quad u_S = \frac{1}{\mu(S)} \int_\Omega u \,d\mu .
\end{equation}
However, we are particularly interested in junctions between several thin structures of possibly different dimension, i.e. when $S$ is a finite union of compact manifolds of possibly different Hausdorff dimensions. In such a case \eqref{point-g} may not be true.

One of the consequences of the global Poincar\'e inequality \eqref{point-g} is the following property:
\begin{equation} \label{connect prop}
u \in H^{1,2}_\mu(\Omega), \quad \nabla_\mu u =0\quad \mu-\text{a.e.} \Rightarrow u=\text{const.} \quad \mu-\text{a.e.}
\end{equation}
In other words
$$
\mathrm{dim}\, \mathrm{ker}\, \nabla_\mu = \mathrm{dim}\{ u \in H^{1,2}_\mu(\Omega) \colon \nabla_\mu u = 0\}= 1,
$$
so $\mathrm{ker}\, \nabla_\mu$ contains only constant functions. In a general situation of a multijunction measure this is not necessarily true, i.e. it may happen that
 $$
\mathrm{dim}\, \mathrm{ker}\, \nabla_\mu >1.
$$
The important feature is ''how big'' the junction area is. Every space that supports a global Poincar\'e inequality is connected. We expect that removing a set of zero capacity from the space should not affect the Poincar\'e inequality. In particular, removing a set of zero capacity should not disconnect the space. Examples show that our expectation need not be true. We consider the situation presented on Figure \ref{antenka}. Since the capacity of a point on a plane is equal to zero, the global Poincar\'e inequality cannot be valid for this set. A similar situation occurs for the set presented on Figure \ref{dwa kola}. We do not want to recall the definition of a (Sobolev) capacity here (we refer the interested reader for instance to the book \cite{bjoerns}, chapter 4). However, in the situation presented on the Figure \ref{antenka}, 
the argument is simple and straightforward.

\begin{figure}          
\centering
\begin{tikzpicture}[scale=0.5,rotate=90]
\begin{scope}
\clip (-4,0) circle [x radius=0.7, y radius=3];
\filldraw[fill=lightgray!50!white] (-5,-3)--(-5,3)--(-3,3)--(-3,-3)--cycle;
\end{scope}

\draw[very thin] (-4,0) circle [x radius=0.7, y radius=3];
\draw[thick] (-4,0) -- (-1,0);
\draw[very thick] (-4,-3) arc [start angle=270, end angle=90, x radius=0.7, y radius=3];

\filldraw (-4,0) circle [radius=0.04];
\end{tikzpicture}
\caption{} \label{antenka}
\end{figure}

\begin{figure}[h]                   
\centering
\begin{tikzpicture}[scale=0.6]

\filldraw[fill=lightgray,opacity=0.7] (-3,0) circle [x radius=3, y radius=0.8];
\filldraw[fill=lightgray,opacity=0.5] (0,3) circle [x radius=0.8, y radius=3];

\draw[thick] (0,0) arc [start angle=270, end angle=90, x radius=0.8, y radius=3];
\draw[thick] (0,0) arc [start angle=360, end angle=180, x radius=3, y radius=0.8];
\filldraw (0,0) circle (2pt);
\end{tikzpicture}
\caption{} \label{dwa kola}
\end{figure}

Consider the unit disk on a plane $D = \{ (x,y) \in \bR^2 \colon x^2 + y^2 < 1\}$ and the function
$$
f(x,y) = -\log \log (1+ (x^2 + y^2)^{-1/2}).
$$
This is an unbounded function which belongs to the Sobolev space $H^{1,2}(D)$. Set
\begin{equation*}
\zeta(t) =
\begin{cases}
0 & t<0,\\
t & t \in [0,1], \\
1 & t > 1.
\end{cases}
\end{equation*}
Then $\zeta \circ f \in H^{1,2}(D)$ as well. Define a sequence of functions $f_n $ as
\begin{equation*}
f_n(x,y) = \zeta(f(x,y)-n+1).
\end{equation*}
Since the $H^{1,2}(D)$ energy of $f$ is finite, the $H^{1,2}(D)$-energy of $f_n$ tends to zero.

\medskip

Now, we consider a metric space $(S,\mu)$, such that $S= X \cup_P Y$, where
$$
X = \{ (x,y,z) \in \mathbb{R}^3 \colon z=0;\ x^2 + y^2 \leq 1 \}
$$
and $\mu|_X$ is a $2$-dimensional Hausdorff measure;
$$
Y = \{ (x,y,z) \in \mathbb{R}^3 \colon x=y=0; z \in [0,1] \}
$$
and $\mu|_Y$ is a $1$-dimensional Hausdorff measure and $P=(0,0,0)$.

\begin{proposition}
 Let us suppose that $S$ and $\mu$ are defined above. Then, $\dim\ker \nabla_\mu >1$.
\end{proposition}
\begin{proof}
 We define a sequence of functions $u_n : S \to \mathbb{R}$ as
\begin{equation*}
u_n(s) =
\begin{cases}
\zeta(f(x,y)-n+1) & s=(x,y,0) \in X;\\
1 & s=(0,0,z)\in Y
\end{cases}
\end{equation*}
and a function
\begin{equation*}
u_o(s) =
\begin{cases}
0 & s=(x,y,0) \in X;\\
1 & s=(0,0,z)\in Y.
\end{cases}
\end{equation*}

Since
$$
\int_S |u_n(s)| \, d\mu = \int_X |u_n(s)| \, d\mathcal{H}^2 + \int_Y |u_n(s)| \, d\mathcal{H}^1
$$
and
$$
\int_S |\nabla_\mu u_n(s)| \, d\mu  = \int_X |\nabla_\mu u_n(s)| \, d\mathcal{H}^2 + \underbrace{\int_Y | \nabla_\mu u_n(s)| \, d\mathcal{H}^1}_{=0},
$$
then it follows that $u_n\in H^{1,2}_\mu$. By the construction, we also have
$$
\| u_n -u_o \|_{H^{1,2}_\mu} \to 0,
$$
which implies $u_o \in H^{1,2}_\mu$ which in turn falsifies the Poincar\'e inequality \eqref{point-g} on $S$ and it shows that \eqref{connect prop} in this situation is not true. As a result, our claim holds.
\end{proof}

\bigskip

Let us denote
$$
V = \ker \nabla_\mu \subset H^{1,2}_\mu \subset L^2(\Omega, \mu).
$$
This is a finite dimensional subspace of $L^2(\Omega, \mu)$.
We define
\begin{equation} \label{P}
P  \colon H^{1,2}_\mu \to V
\end{equation}
as be the orthogonal projection onto $V$ (with respect to $L^2$ scalar product).
Set $d=\dim V$.  We can split the set $\{ 1,\ldots,J \}$ from the condition \textbf{[S]} into a final family of disjoint sets of indices
$$
\{ 1,\ldots,J \} = J_1 \cup \ldots \cup J_d, \quad d\leq J, \quad J_l \cap J_m = \emptyset,\quad l\neq m.
$$
The sets $J_l$, $l=1,\ldots, d$ are defined in the following way. The set
$$
\widetilde{S}_l = \bigcup_{i \in J_l} S_i
$$
is such that
$$
\chi_{\widetilde{S}_l} \in \ker \nabla_\mu,
$$
and $J_l$ are "maximal" in the sense that  for any proper subset $K \subset J_l$,
$$
\text{the characteristic function of } \bigcup_{i \in K } S_i  \text{ does not belong to $V$}.
$$

Then the projection $P$ in \eqref{P} can be expressed as
\begin{equation}\label{P formula}
P = \sum_{l=1}^d P_l, \quad \text{where} \quad P_l \, u= \chi_{\widetilde{S}_l} \mvint_{\widetilde{S}_l} u \,  d\mu.
\end{equation}

For a multijunction measure satisfying \textbf[S] the following  weaker version of the Poincar\'e inequality holds.
\begin{theorem}\label{tw-p}
Let us suppose that a multijunction measure satisfies \textbf[S]. Then there is $C_l>0$ such that for
$u\in H^{1,2}_\mu$
\begin{equation} \label{poincare2}
\sum_{i\in J_l} \int_\Omega | u - P_l u |^2 \, d\mu_i  \le C_l
\sum_{i\in J_l} \int_\Omega | \nabla_\mu u |^2 \, d \mu_i,  \quad \text{for $l=1, \ldots, d$},
\end{equation}
where $P_l$ is defined in \eqref{P formula}.
\end{theorem}

\medskip

Let us mention in advance that, the  regularity of solutions depends in an essential way  on whether the set $S$ supports the inequality \eqref{point-g} or  \eqref{poincare2}.
It is known that if a metric measure space (the set $S$ may be treated as such) supports the Poincar\'e inequality \eqref{point-g}, one can expect solutions to be continuous. The weaker version is not enough to ensure this, as shown on examples presented in Section \ref{examples}.




\section{The elliptic problem} \label{elliptic problem}

We would like to use the calculus of variations to establish existence of solutions to (\ref{r1}) in the space $H^{1,2}_\mu
$. For this purpose we recall the definition  of $E(u)$ for
$H^{1,2}_\mu
.$
We have to specify the assumptions on $A$ and $Q$.

In principle, $u$ is an element of $L^2(\Omega,\mu)$, so we must specify properties of $Q$ appropriately. This is why we make the following assumption,
\begin{equation}\label{z2}
\begin{split}
& \text{$Q$ is a measure of a finite total variation supported in $\Omega$,}\\
 &\text{absolutely continuous with respect to $\mu$,}\\
 &\text{and there is  $f \in L^2(\mu)$ such that } Q = f \mu.
 \end{split}
\end{equation}
This assumption in particular makes elements of $H^{1,2}_\mu$ measurable with respect to $Q$. 


Now, we specify our assumptions on $A$. We consider 
\begin{equation}\label{Alin}
A\in L^\infty(\Omega; M(N\times N),\mu)\hbox{ and for a.e. $x$ the matrix $A(x)$ is symmetric.}
\end{equation}

We have to present the positivity of $A$ in a way suitable for dealing with $\nabla_\mu u$ in the tangent space to $\mu$. We require that at $\mu$-almost every $x\in \Omega$, we have
\begin{equation}\label{z3}
 (A(x) \xi, \xi) \ge \la |\xi|^2,\qquad \hbox{for all }\xi \in T_\mu(x). 
\end{equation}

After these preparations, we state our first result.
\begin{proposition}\label{pr3.1}
 Let us suppose that  $A$ satisfies (\ref{Alin}) and for each $x\in\Omega$ matrix $A(x)$ is
 non-negative. Suppose further that the measure $\mu$ satisfies the assumptions \textbf{[S]}. Then, the relaxation of $E$, defined by (\ref{dE}), is given by
 $$
 E_\mu(u) = \frac12 \int_\Omega (A_\mu(x) \nabla_\mu u, \nabla_\mu u)\,d\mu - \int _\Omega fu\, d\mu.
 $$
Here, matrix $A_\mu(x)$ is defined by
\begin{equation}\label{rA}
 A_\mu(x) = 
 (I - \Pi(x,\cdot))A(I - \Pi(x,\cdot)),
\end{equation}
where $(I - \Pi(x,\cdot))$ is the orthogonal projection of $\bR^N$ onto $T_\mu(x)$,
hence $A\mu(\cdot)$ is $\mu$-measurable.
\end{proposition}
\begin{remark}
 In other words, $A_\mu(x)$ is the projection of $A(x)$ onto the tangent space $T_\mu(x)$.
\end{remark}
\begin{proof}
{\it Step 1.} Let us set $F^B(x,p) :=(B(x)p,p)$,
where $B\in L^\infty(\Omega,M(N\times N),\mu)$ and for $\mu$-a.e. $x\in \Omega$ matrix $B(x)$ is symmetric and positive definite, then $E^B(u)= \frac12 \int_\Omega F^B(x,\nabla u)\,d\mu$.
By \cite{energies}, $E^B_\mu$,   the relaxation of $E^B$, is given by
$$
E^B_\mu(u) = \frac12 \int_\Omega F^B_\mu(x,\nabla_\mu u)\,d\mu,
$$
where
$$
  F^B_\mu(x,p) = \inf \{ F^B(x,p+ \xi): \quad \xi \in (T_\mu (x))^\perp\}.
$$
Let us suppose that $e_i$, $i=1,\ldots, N-k$ span $(T_\mu (x))^\perp$. We may choose them so that
$$
(B e_i, e_j) = 0,\quad i\neq j, \qquad (B e_i, e_i)>0, \quad i=1,\ldots,
N-k.
$$
Then, the minimization problem in the definition of $F_\mu$ may be written as
$$
\inf \{ F^B(x,p + \sum_{i=1}^{N-k}t_i e_i):\ t_i\in\bR, \ i =1,\ldots, N-k\}.
$$
A simple differentiation yields optimality conditions
$$
t^o_i = -(Bp, e_i)/ (B e_i, e_i), \quad i=1,\ldots, l.
$$
We easily see that
$$
F^B_\mu(x,p) = (B_\mu(x)p,p),
$$
where $B_\mu(x)$ is given by
\begin{equation}\label{rB}
 B_\mu(x) = B(x) - \sum_{i=1}^{N-k} \frac{B(x) e_i(x)\otimes B(x) e_i(x)}{(B(x) e_i(x), e_i(x))},
\end{equation}
where $e_i$, $i=1,\ldots,N-k$ are linearly independent and span $(T_\mu(x))^\perp$. It is easy to check that
$$
B_\mu(x) = (I - \Pi(x,\cdot))B (I - \Pi(x,\cdot)).
$$
Hence, due to measurability of projection $\Pi$ the mapping $x\mapsto B_\mu(x)$ is $\mu$-measurable.

{\it Step 2.} We assumed that $A$ satisfies (\ref{z3}), so if we set $F(x,p)= (A(x)p,p)$, then this intergrand does not satisfy the lower estimate
$$
F(x,p)\ge c_0 |p|^2
$$
for any $c_0>0$, which is in the assumptions of \cite[Theorem 3.1]{energies}. For this reason we take any symmetric matrix $W\in L^\infty(\Omega, M(N \times N),\mu)$, whose kernel at $x$ is the image of $A(x)$. Moreover, we require that there is  $c_0>0$ such that
$$
(W(x)\xi,\xi) \ge c_0 |\xi|^2,
$$
for all $\xi \in \ker A(x)$.
Then, $B(x) = A(x) + W(x)$ is symmetric and for $\xi$ in the image of $A(x)$ and $\zeta\in\ker A(x)$ we have
$$
(B(x),\xi+\zeta,\xi+\zeta)  
= (A(x)\xi,\xi) +
( W(x)\zeta,\zeta) ,
$$
because the image of a symmetric matrix is perpendicular to its kernel. Hence,
$$
(B(x),\xi+\zeta,\xi+\zeta) \ge \lambda |\xi|^2 + c_0 |\zeta|^2 \ge \min\{\lambda, c_0\}
|\xi+\zeta|^2.
$$
and formula (\ref{rB}) applies to $B= A+W$.

{\it Step 3.} We claim that if $B= A+W$, then $B_\mu = A_\mu$, where $A_\mu$ is given by (\ref{rA}). Indeed, 
the summation in (\ref{rB}) may be split, the first $l\le N-k$ vectors belong to $\image A\cap (T_\mu)^\perp,$ while vectors $e_i$, $i=l+1,\ldots, N-k$ span the kernel of $A$. Then, $B_\mu$ takes the following form,
$$
B_\mu = (A +W) - \sum_{i=1}^l\frac{A e_i\otimes A e_i}{(A e_i, e_i)}
 - \sum_{i=l+1}^{N-k} \frac{W e_i\otimes W e_i}{(W e_i, e_i)} = A_\mu + W_\mu.
$$
It is easy to see that for any $\xi\in \bR^N$ we have $B_\mu\xi \in T_\mu$. Moreover, one can check that if $\zeta\in \ker A$, then $W_\mu \zeta =0 $ and if $\xi\in \image A$, then $W_\mu \xi =0$. Thus, our claim follows.

{\it Step 4.} We will check that the lower semicontinuous envelope of $E$ is $E_\mu$.

By the definition we have $E(u) \le E^B(u)$. If we denote by bar the lower semicontinuous envelope, then we see,
$$
\bar E(u) \le \bar  E^B(u) = E^B_\mu (u).
$$
If we look at the definition of $E_\mu$, then we see that $E_\mu(u) \le E(u)$. Moreover, $E_\mu$ is lower semicontinuous, hence
$$
E_\mu(u)\le  \bar E(u) \le \bar  E^B(u) = E^B_\mu (u) = E_\mu(u) ,
$$
where the last equality follows from step 3. Finally,
$$
E_\mu(u) =\bar E(u).
$$
\end{proof}

\bigskip
We have to describe $\ker \di A_\mu \nabla_\nu$, what is necessary, before we can specify $Q$. Actually, we prove:
\begin{proposition}
 $\ker \di A_\mu \nabla_\nu =  \ker \nabla_\nu$.
\end{proposition}
\begin{proof}
The inclusion $\supset$ is obvious. Let us suppose that $u \in H^{1,2}_\mu$ satisfies
$$
\di (A_\mu \nabla_\mu u \mu) = 0\qquad\hbox{in }\cD'(\bR^N).
$$
In other words,
$$
\int_\Omega A_\mu\nabla_\mu u \nabla \vfi \,d\mu= 0\qquad\hbox{for all }\vfi\in \cD(\bR^N).
$$
We notice that $\nabla \vfi (x) = \nabla_\mu \vfi(x) + (Id - \Pi_\mu(x, \nabla \vfi))$ for $\mu$-a.e. $x\in S$. Since $A_\mu \nabla_\mu u(x) \in T_\mu(x)$ for $\mu$-a.e. $x\in S$, we have
$$
0= \int_\Omega A_\mu\nabla_\mu u \nabla \vfi \,d\mu=
\int_\Omega A_\mu\nabla_\mu u \nabla_\mu \vfi \,d\mu
$$
We may take a sequence $\vfi_n\in \cD(\bR^N)$ such that $\nabla_\mu \vfi_n $ converges to
$\nabla_\mu u$. This yields,
$$
0= \int_\Omega A_\mu\nabla_\mu u \nabla_\mu u\,d\mu.
$$
Combining this with ellipticity of $A_\mu$, see (\ref{z3}), yields the claim.
\end{proof}

\medskip

After these preparations we can state one of our main results.

\begin{theorem}\label{tw1}
Let us assume that the conditions:  \textbf{[S]} on $\mu$ and (\ref{z2}) on $Q$ as well as (\ref{z3}) on $A$ hold.
In addition, we assume
that $\langle Q, h \rangle =0$ for all $h$ in the kernel of $\nabla_\mu$. Then, there exists a minimizer $u$ of the functional $E_\mu$ defined on the linear subspace
$$
H= \{ u \in H^{1,2}_\mu: \quad Pu =0\},
$$
where $P$ is the orthogonal projection onto $\ker \nabla_mu$ described by \eqref{P formula}.
\end{theorem}
\begin{proof}
We have already computed $E_\mu$, the relaxation of $E$. Actually, we study minimizers of $E_\mu$.
Let us suppose that $\{ u_k\}\subset H$ is a sequence minimizing $E_\mu$, i.e.,
$$
\lim_{k\to \infty} E_\mu(u_k) = \inf\{ E_\mu(v):\ v \in H
\} =:R.
$$
We may assume that for all $k\in \bN$ we have,
$$
R\le \frac12  \int_\Omega (A_\mu(x)\nabla_\mu u_k, \nabla_\mu u_k)\,d \mu - \int_\Omega fu_k\, d\mu \le R+1.
$$
Due to (\ref{z3}), we end up with,
\begin{equation}\label{r3}
 \frac{\lambda}{2} \int_\Omega | \nabla_\mu u_k|^2 \,d\mu - \int_\Omega fu_k\, d\mu \le R+1.
\end{equation}
Now, the Theorem \ref{tw-p} and $P u_k =0$ yield that there exists $C>0$ such that
$$
\int_{\Omega} (u^2_k + | \nabla_\mu u_k|^2) \,d\mu \le C \int_{\Omega} | \nabla_\mu u_k|^2 \,d\mu.
$$
Combining this with (\ref{r3}) and Young's inequality yields,
$$
(1-\epsilon\frac{C}{2\lambda}) \| u_k \|^2_{H^{1,2}_\mu}\le \frac{C}{\lambda 2\epsilon} \| f \|^2_{L^2(\Omega, \mu)} +R+1.
$$
In other words the minimizing sequence is bounded in $H^{1,2}_\mu
$. Due to the results of \cite{energies} we deduce that there is a weakly convergent subsequence (not relabeled), with limit $u$. Now, we invoke the lower semicontinuity results so that we deduce,
$$
\varliminf_{k\to\infty} E_\mu(u_k) \ge \frac12 \int_\Omega (A_\mu(x)\nabla_\mu u, \nabla_\mu u)\,d \mu - \int_\Omega fu\, d\mu.
$$
Uniqueness of minimizers will be treated separately. 
\end{proof}

Establishing the Euler-Lagrange equations requires further assumptions on $A_\mu$ and $\mu$.

\begin{proposition}\label{pr-EL}
 Let us suppose that $u$ is a minimizer of $E_\mu$ on $H^{1,2}_\mu
 $, then the following weak form of Euler-Lagrange equation holds,
\begin{equation}\label{r-EL}
 \int_\Omega (A_\mu \nabla_\mu u, \nabla_\mu \vfi)\,d\mu -\int_\Omega f \vfi\, d\mu =0\qquad\hbox{for all }
 \vfi \in C^1_0(\bR^N).
\end{equation}
\end{proposition}
\begin{proof}
 Since $u$ is a minimizer, then for any test function $\vfi\in H^{1,2}_\mu
 $, (in particular we may take $\vfi \in C^1_0(\bR^N)$), we have
 $$
 E_\mu(u) \le E_\mu(u + \vfi).
 $$
Thus,
$$
0 \le \int_\Omega (A_\mu \nabla_\mu u, \nabla_\mu \vfi)\,d\mu +  \frac12 \int_\Omega (A_\mu \nabla_\mu \vfi, \nabla_\mu \vfi)\,d\mu -
\int_\Omega f \vfi\, d\mu .
$$
After replacing $\vfi$ with $t \vfi$, where $t$ is real, we will deduce that (\ref{r-EL}) holds.
\end{proof}

\begin{corollary}\label{wn-jed} Let us suppose that the assumptions of the previous theorem hold.
If $u_1$ and $u_2$ are two minimizers of $E_\mu$, which are perpendicular to ker$\nabla_\mu$, then $u_1 = u_2$.
\end{corollary}
\begin{proof}
Let us take $u = u_2 - u_1$ and take the difference of (\ref{r-EL}) corresponding to $u_1$ and $u_2$,
$$
\int_\Omega (A_\mu \nabla_\mu u, \nabla_\mu \vfi)\,d\mu =0.
$$
Since $u = u_2 - u_1$ can be approximated in the $H^{1,2}_\mu$-norm by a sequence of $C^1$ functions, we may take $\vfi = u $ in the identity above,
$$
0 = \int_\Omega (A_\mu \nabla_\mu u, \nabla_\mu u)\,d\mu \ge \lambda \| \nabla_\mu u\|_{L^2_\mu},
$$
i.e. $u \in \ker \nabla_mu$. Since we assumed that $u \in (\hbox{ker}\,\nabla_\mu )^\perp$, we deduce that $u = 0$, i.e. $u_2 = u_1$.
\end{proof}

We do not specify any boundary conditions so we expect that the solution satisfies the so-called natural boundary condition, however, in a weak form suitable for $H^{1,2}_\mu$.

\bigskip

\begin{remark}
Let us remark that in the case $S$ is a smooth compact manifold of codimension 1,  $\mu = \cH^{N-1} \llcorner S$, and $A$ is the identity matrix, the equation \eqref{r1} takes the familiar form of 
$$
\Delta u +f=0 \quad \text{on $S$},
$$
where $\Delta$ is the Laplace-Beltrami operator on $S$.
\end{remark}

\section{On the boundary condition} \label{bd condition}

After \cite{frid}, we introduce the notion of a vector valued measure with bounded divergence.

\begin{definition}(see \cite[Definition 1.1]{frid})\label{d-bdiv}
 Let us suppose that $F$ is a vector-valued Radon measure. We set
 $$
 |\di F|(\Omega) :=\sup\{ \langle F,\nabla \psi\rangle:\ \psi\in C^1_c(\Omega), |\psi(x)|\le 1\}.
 $$
 If $|\di F|(\Omega)$ is finite, we say that the {\it divergence of $F$ has a finite total variation}.
\end{definition}
The advantage of measures whose divergence  has a finite total variation is that one can define the trace of their normal component of $\d\Omega$. We recall, see \cite[Theorem 2.2]{frid},
\begin{proposition}\label{pr-sl}
We assume that the boundary of $\Omega$ is  smooth. If both $F$ and its divergence have a bounded total variation, there exists a continuous linear functional $[F,\nu]|_{\d\Omega}$ over $C^2(\d\Omega)$ such that
$$
\langle [F,\nu]|_{\d\Omega}, \vfi\rangle =
\langle \di F, \vfi \rangle + \langle F, \nabla\vfi \rangle
\quad
\text{for any $\vfi \in C^2(\overline{\Omega})$}.
$$
\end{proposition}
Actually, the original statement of the Proposition is more general, in particular less smoothness of the boundary of $\Omega$ is sufficient. Namely, it suffices if $\d\Omega$ is Lipschitz deformable, see \cite[Definition 2.1]{frid}. In addition,
$[F,\nu]|_{\d\Omega}$ could be defined
as a functional on  the space $\Lip(\gamma,\Omega)$, $\gamma>1$, which is larger than $C^2(\d\Omega)$, see \cite{frid}. However, such generality is not necessary here, so we choose the simplicity of exposition.

We want to show that if $u$ is a minimizer of $E$, then $F:= A_\mu  \nabla_\mu u$ has divergence with a finite total variation. 

\begin{theorem} \label{tw b}
Let us suppose that the structural assumption \textbf{[S]} as well as the condition \eqref{z2} hold. If  $u$ is the minimizer of $E_\mu$, then $F:= A_\mu  \nabla_\mu u\in \cM(\Omega, \bR^N)$ has divergence with finite total variation. As a result, $[A_\mu  \nabla_\mu u, \nu]|_{\d\Omega} =0$ in the sense of Proposition \ref{pr-sl}.
\end{theorem}
\begin{proof}
We have to reconcile Definition \ref{d-bdiv} with the Euler-Lagrange equation (\ref{pr-EL}): in  (\ref{pr-EL}) the inner product of $A_\mu \nabla_\mu u$ with $\nabla_\mu u$ is taken, while in the definition  $A_\mu \nabla_\mu u$ is multiplied with the full gradient of the test function. 
The relaxation result, Proposition \ref{pr3.1}, makes this difference apparent. In fact, no matter what the matrix $A(x)$ is, Proposition \ref{pr3.1} implies that
\begin{equation}\label{z5}
 A_\mu: T_\mu \to T_\mu.
\end{equation}
Observation
(\ref{z5}) makes (\ref{r-EL}) coincide with
$$
\langle F, \nabla \psi\rangle = \langle Q,  \psi\rangle \quad \forall  \psi\in C^1_c(\Omega).
$$
Indeed, if $\Pi_\mu(x)$ is the projection onto $T_\mu(x)$ at each $x\in\Omega$, then for $\psi\in C^1_c(\Omega)$ we have $\nabla\psi (x)= \Pi_\mu(x)\nabla \psi(x) + (Id-\Pi_\mu(x) )\nabla \psi(x)$. As a result we notice,
\begin{eqnarray*}
\langle F, \nabla \vfi\rangle &=& \int_\Omega ( A_\mu \nabla_\mu  u, \nabla_\mu \psi  + (Id- \Pi_\mu) \nabla \psi)\,d\mu \\
&=& \int_\Omega (A_\mu \nabla_\mu u, \nabla_\mu \psi) \, d\mu + 0 \\
&=& \langle Q, \psi\rangle,
\end{eqnarray*}
where we used (\ref{z5}).
Now, we can see that $|\di F|(\Omega)$ is finite, because
$$
\sup\{ \langle Q,\psi \rangle:\ \psi \in C^1_c(\Omega), |\psi(x)|\le 1\} = \int_\Omega |f|\,d\mu <\infty.
$$
Thus, we immediately deduce that $[A_\mu \nabla_\mu u,\nu]|_{\d\Omega}$ exists in the sense explained in
Theorem \cite[Theorem 2.2]{frid}. Now, we have to show that it vanishes., i.e.
$$
\langle \di A_\mu \nabla_\mu u, \vfi \rangle + \langle A_\mu \nabla_\mu u, \nabla\vfi \rangle = 0
\quad
\text{for any $\vfi \in C^2(\overline{\Omega})$}.
$$
Let us suppose that $\eta_k\in C^\infty_c(\Omega)$ are the cut-off functions such that
$\eta_n(x)=0$ for $x\in\Omega$ at a distance smaller than $1/(k+1)$ from $\d\Omega$ and equal to 1 for $x$ further than $1/k$ from $\d\Omega$. Then, since $\eta_k \vfi \in C^1_c (\Omega)$ and by using of the  Euler-Lagrange equation (\ref{r-EL}) with $(1-\eta_k)\vfi$ as a test function, we obtain
\begin{align*}
\langle & \di  A_\mu  \nabla_\mu u, \vfi \rangle + \langle A_\mu \nabla_\mu u, \nabla\vfi \rangle \\
    &= \langle \di A_\mu \nabla_\mu u, \eta_k \vfi \rangle +\langle \di A_\mu \nabla_\mu u, (1-\eta_k) \vfi \rangle
    +  \langle A_\mu \nabla_\mu u, \nabla(\eta_k \vfi) \rangle + \langle A_\mu \nabla_\mu u, \nabla( (1-\eta_k) \vfi) \rangle \\
    &= \langle \di A_\mu \nabla_\mu u, (1-\eta_k) \vfi \rangle + \langle A_\mu \nabla_\mu u, \nabla( (1-\eta_k) \vfi) \rangle \\
    &= \langle \di A_\mu \nabla_\mu u, (1-\eta_k) \vfi \rangle + \int_\Omega f (1-\eta_k) \vfi\, d\mu.
\end{align*}

If $\omega$ is a measure with a finite variation, then
$$
\lim_{k\to\infty}\int_\Omega (1-\eta_k)\vfi \,d\omega =0,
$$
because of the Lebesgue dominated convergence Theorem. If we apply this remark to $\omega = Q$ and
$\omega =\di A_\mu \nabla_\mu u$, then we deduce that
$$
\langle [ A_\mu \nabla_\mu u, \nu]_{\d\Omega},\vfi\rangle =0
$$
for all $\vfi \in C^2(\bar\Omega)$. Our claim follows.
\end{proof}

\begin{remark}\label{rem41}
Theorem \ref{tw b} tells us about the meaning of $[A_\mu \nabla_mu, \nu]$ only at the boundary of $\Omega$. So, if a part of $\d S$ does not touch $\d\Omega$, then we have no information about the boundary behavior of $u$. This is why we impose the third condition of \textbf{[S]}. On the other hand,  the fourth condition in \textbf{[S]} rules out spurious conditions.
\end{remark}

\section{Examples of Neumann boundary problems} \label{examples}

We present below several examples of Neumann boundary problems on various sets $S$. The first two examples deal with situations when the Poincar\'e inequality \eqref{point-g} holds. We further present an example (see Proposition \ref{discont}) showing that discontinuous solutions actually occur, even for regular data, when the set $S$ supports only the weaker inequality \eqref{poincare2}.

\bigskip

Let us suppose that $\Omega$ is a ball centered at 0 with radius 1, $T$  is an inscribed isosceles triangle and $S_1$, $S_2$, $S_3$ is the set of radii connecting the center of $\Omega$ with the vertices of $T$ (see Fig. \ref{litera Y}). We set $S = S_1\cup S_2\cup S_3$. Moreover, $\mu = \cH^1 \llcorner S$. We set $Q = \sum_{i=1}^3 a_i\chi_{S_i}$. 

\begin{figure}[h]      
\centering
\begin{tikzpicture}[scale=0.7]

\draw[ultra thick] (-2.55,1.2) -- (0,0);
\draw[ultra thick] (2.55,1.2) -- (0,0);
\draw[ultra thick] (0,0) -- (0,-2.85);

\filldraw (-2.55,1.2) circle (2pt);
\filldraw (2.55,1.2) circle (2pt);
\filldraw (0,-2.85) circle (2pt);

\filldraw[fill=lightgray!70!white,opacity=0.7] (0,0) circle [radius=2.85];

\end{tikzpicture}
\caption{} \label{litera Y}
\end{figure}

We can state the following fact.

\begin{proposition}
Let us assume that $S$, $\mu$ and $Q$ are defined above and $a_1 +a_2+ a_3 =0$.
If $A = Id$, then the relaxed form of equation (\ref{r1}) takes the form,
\begin{equation}\label{rT}
 \frac{d^2 u}{ds^2} = -a_i\qquad\hbox{in the interior of } S_i, \qquad i=1,2,3
\end{equation}
and
\begin{equation}\label{rN}
\frac{d u}{ds} = 0\qquad\hbox{on } S_i\cap \d\Omega, \qquad i=1,2,3.
\end{equation}
Here, $s$ is the arc-length parameter counted from the center of $\Omega$.

Finally, the solution of (\ref{rT}-\ref{rN}) with $\int_S u\,d \cH^1 =0$ has the following form,
$$
u = -\sum_{i=1}^3 a_i(\frac{s^2}{2}-s)\chi_{S_i}.
$$
\end{proposition}
\begin{proof}
 By Proposition \ref{pr-roz} $T_\mu(x) = T_{S_i}(x)$, except for $x$ being the  center of the ball. Hence, Proposition \ref{pr3.1} yields $\di A_\mu \nabla_\mu u = \frac{d^2 u}{ds^2}$ in the interior of $S_i$, $ i=1,2,3$. Also the form, (\ref{rN}), of boundary data follows.

 The assumption  $a_1 +a_2+ a_3 =0$ means that $Q$ is perpendicular to the kernel of $\nabla_\mu$ so that we may apply Theorem \ref{tw0}.

 The form of any solution to  (\ref{rT}-\ref{rN}) follows from solving the ODE's (\ref{rT}). While doing so, we keep in mind the requirement that the solution must be continuous at $s=0$, i.e. at the center of the ball $\Omega$. The solution is determined up to an additive constant, which can be computed due to the zero average condition, $\int_S u\, d\cH^1=0$.
\end{proof}

We can also study (\ref{r1}) on sums of manifolds of dimension $k\geq 1$.
We describe the set $S$ depicted on Fig. \ref{kartki2}.

\begin{figure}[h]   
\centering
\begin{tikzpicture}[scale=0.7]
\begin{scope}
\clip (0,0) circle [x radius=3, y radius=1];
\filldraw[fill=lightgray!50!white] (-1,-1)--(-5,-1)--(-3,1)--(1,1)--cycle;
\end{scope}
\begin{scope}
\clip (0,0) circle [x radius=1, y radius=3];
\filldraw[fill=lightgray!50!white] (-1,-1)--(-1,-5)--(1,-3)--(1,1)--cycle;
\end{scope}
\begin{scope}
\clip (0,0) circle [x radius=3, y radius=1];
\filldraw[fill=lightgray!80!white,opacity=0.7] (-1,-1)--(4,-1)--(6,1)--(1,1)--cycle;
\end{scope}
\begin{scope}
\clip (0,0) circle [x radius=1, y radius=3];
\filldraw[fill=lightgray!80!white,opacity=0.7] (-1,-1)--(-1,4)--(1,6)--(1,1)--cycle;
\end{scope}
\draw[very thin] (0,0) circle [x radius=3, y radius=1];
\draw[very thin] (0,0) circle [x radius=1, y radius=3];
\shadedraw[ball color=white,opacity=0.3] (0,0) circle [radius=3];
\draw[thick] (-3,0) arc [start angle=180, end angle=360, x radius=3, y radius=1];
\draw[thick] (0,-3) arc [start angle=270, end angle=90, y radius=3, x radius=1];
\end{tikzpicture}
\caption{} \label{kartki2}
\end{figure}

\begin{proposition}
 We assume that $A = Id$, $\Omega$ is the unit ball $B(0,1)\subset\bR^3$ and $S = D_1\cup D_2$ are two great disks of $B(0,1)$,
$$
D_1 = \{ (x_1, x_2, x_3)\in B(0,1): x_3 =0\},\qquad
D_2 = \{ (x_1, x_2, x_3)\in B(0,1): x_1 =0\}.
$$
Moreover, $Q = Q_1\chi_{D_1} + Q_2\chi_{D_2}$, where $\int_{D_i}dQ_i =0$, $i=1,2$.
Then, the relaxed equations (\ref{rel-r1}) take the form
$$
\begin{array}{ll}
 \frac{\partial^2 u}{\partial x_1^2} + \frac{\partial^2 u}{\partial x_2^2} + Q_1 = 0& \hbox{in } D_1,\\
 \frac{\partial u}{\partial\nu_1} = 0& \hbox{on } \partial D_1
\end{array}
$$
and
$$
\begin{array}{ll}
 \frac{\partial^2 u}{\partial x_2^2} + \frac{\partial^2 u}{\partial x_3^2} + Q_2 = 0& \hbox{in } D_2,\\
 \frac{\partial u}{\partial\nu_2} = 0& \hbox{on } \partial D_2,
\end{array}
$$
where $\nu_i$ are normal to $D_i$ in the plane containing $D_i$, $i=1,2$. The  condition of orthogonality to the kernel of $\nabla_mu$ is $\int_{D_i} u\,d\cH^2 =0$, $i=1,2$.
\end{proposition}
\begin{proof}
We easily see that, due to Proposition \ref{pr3.1}, we will have
$$
A_\mu = (e_1\otimes e_1 +e_2\otimes e_2)\chi_{D_1} + (e_3\otimes e_3 +e_2\otimes e_2)\chi_{D_2}.
$$
Hence, the form of relaxed equations follows. The boundary conditions are addressed in Corollary \ref{pa-pa}.
\end{proof}

We notice that
$A_\mu$ depends upon $x\in \Omega$  despite $A$ being constant. Moreover, the kernel of $A_\mu$ is two-dimensional.

\bigskip

Due to the generalized Poincar\'e inequality \eqref{poincare2},
we can study (\ref{r1}) on sums of manifolds of different dimensions.
We analyze an  example of an equation on the domain depicted on
Figure \ref{tloczek}. We define the design region $\Omega$ to be $B(0,\sqrt2)\subset\bR^3$. We set,
$$
S_i = \Omega\cap\{(x_1,x_2,x_3)\in \bR^3:\ x_1 = (-1)^i\}, \quad i =1,2,\qquad
S_3 = \{(x_1,0,0)\in \bR^3:\ |x_1|\le 1\}.
$$
In the case of this domain, we do not expect the global Poincar\'e inequality \eqref{point-g} to hold.

\begin{figure}[h] 
\centering
\begin{tikzpicture}[scale=0.6]

\begin{scope}
\clip (-4,0) circle [x radius=0.4, y radius=3];
\filldraw[fill=lightgray!50!white] (-5,-3)--(-5,3)--(-3,3)--(-3,-3)--cycle;
\end{scope}
\begin{scope}
\clip (4,0) circle [x radius=0.4, y radius=3];
\filldraw[fill=lightgray!50!white,opacity=0.7] (3,-3)--(5,-3)--(5,3)--(3,3)--cycle;
\end{scope}

\draw[very thin] (-4,0) circle [x radius=0.4, y radius=3];
\draw[very thin] (4,0) circle [x radius=0.4, y radius=3];
\draw[thick] (-4,0) -- (3.6,0);

\shadedraw[ball color=white,opacity=0.2] (0,0) circle [x radius=6, y radius=4.1];
\draw[very thick] (-4,-3) arc [start angle=270, end angle=90, x radius=0.4, y radius=3];
\draw[very thick] (4,-3) arc [start angle=270, end angle=90, x radius=0.4, y radius=3];

\filldraw (-4,0) circle [radius=0.05];
\filldraw (4,0) circle [radius=0.05];
\end{tikzpicture}
\caption{} \label{tloczek}
\end{figure}

\begin{proposition} \label{discont}
We assume that  $A = Id$ and
$$
\mu  = \cH^2 \llcorner (S_1\cup S_2) + \cH^1 \llcorner S_3.
$$
We
assume that $Q = Q_1\chi_{S_1} + Q_2\chi_{S_2}$, i.e. $Q|_{S_3} \equiv 0$.
We will require
that
$$
\int_{S_i} Q_i \,d\mu =0,\quad i=1,2.
$$
We set
$$
Q_1(x_1, x_2,x_3) = q(\sqrt{x_2^2 + x_3^2})\quad\hbox{and} \quad
Q_2(x_1, x_2,x_3) = -q(\sqrt{x_2^2 + x_3^2}).
$$
Then, $u$, the solution to (\ref{r1}) in $S$ with $\int_{S_i} u \, d\cH^2 =0$, $i=1,2$ and $\int_{S_3} u \, d\cH^2 =0$ is discontinuous.
\end{proposition}
\begin{proof} We notice that the relaxation of the functional (Proposition \ref{pr3.1}) yields,
$$
A_\mu = (e_2 \otimes e_2 + e_3 \otimes e_3)\chi_{S_1\cup S_2} + (e_1 \otimes e_1) \chi_{S_3}.
$$
Thus, eq. div$A_\mu\nabla_\mu u + Q=0$ becomes
\begin{equation}\label{r-above}
\frac{\partial^2 u}{\partial x_2^2} + \frac{\partial^2 u}{\partial x_2^3} + q =0 \qquad \hbox{on } S_1,
\end{equation}
$$
\frac{\partial^2 u}{\partial x_2^2} + \frac{\partial^2 u}{\partial x_2^3} - q =0 \qquad \hbox{on } S_2
$$
and
$$
\frac{\partial^2 u}{\partial x_1^2} =0 \qquad \hbox{on } S_3.
$$

The radial symmetry of $q$ implies that
\begin{equation}\label{ex-tl}
	0 = \int_{\{x_2^2+x_3^2\le 1\}} Q_1 \,d\mu = 2\pi \int_0^R r q(r) \,dr.
\end{equation}
The above equation (\ref{r-above}),  considered on $S_1$, due to the radial symmetry takes the following form
$$
u_{rr} + \frac 1r u_r + q =0.
$$
After multiplying by $r$ and integration, we obtain,
$$
r u_r (r) = a - \int_0^r s q(s)\,ds.
$$
In order to solve it, we have to pay a bit of attention to the boundary conditions. Since
$\nu = 2^{-1/2}(-1, x_2,x_3)$ for $(-1, x_2,x_3)\in \partial S_1$, then
the Neumann boundary condition
$$
[A_\mu \nabla_\mu u , \nu ]|_{\d\Omega}= 0
$$
takes the form
$$
(0, u_{x_2}, u_{x_3}) \cdot \nu = \nabla_\mu u \cdot {\bf n} =0,
$$
where $ {\bf n}$ is the outer normal to $S_1$ in the plane $x_1 = -1$. We notice that due to (\ref{ex-tl}),
the above boundary conditions are automatically satisfied, provided that $a=0$. Moreover, $a=0$
is necessary for $u$ to be an element of  $H^{1,2}_\mu$.
The formula for the solution is as follows,
$$
u(r) = b - \int_0^r \frac 1\rho \int_0^\rho sq(s)\,dsd\rho.
$$
The zero mean condition imposed on the solution on $S_1$ implies that
$$
b = \frac 2{R^2} \int_0^R r
\int_0^r \frac 1\rho \int_0^\rho sq(s)\,dsd\rho dr.
$$
We may choose $q$ so that $b>0$. Due to the symmetry of $Q$, we deduce that
$u(-1,x_2,x_3) = -u(1,x_2,x_3) $.
Thus, in particular $u(-1,0,0) = b = - u(1,0,0)$.

Since we set $Q|_{S_3} =0$, then we study the problem of minimizing
$$
\int_{-1}^1 u_{x_1}^2\,dx,
$$
 on $S_3$,
where $u \in H^{1,2}(-1,1)$ and $\int_{-1}^1 u\, dx_1 =0$. As a
result  $u|_{S_3} =0$.

We conclude that we constructed a discontinuous solution when the forcing term is continuous.
\end{proof}

Here is another example.
\begin{proposition}
We take
$$
\Omega = B(0,\sqrt 2), \qquad S_1 = \{ x\in\Omega; x_3 = -1\}, \quad S_2 = \{ x\in\Omega; x_1 = 1\},
$$
$\mu = \cH^2 \llcorner (S_1\cup S_2)$ and
$$
A = e_2 \otimes e_2 + \frac 12 (e_1+e_3) \otimes (e_1+e_3) .
$$
Then,
$$
A_\mu = e_2 \otimes e_2 + \frac 12  e_1 \otimes e_1 \chi_{S_1} + \frac 12  e_3 \otimes e_3 \chi_{S_2}
$$
and the boundary conditions take the following form,
\begin{equation}\label{dwa1}
 0 = [A_\nu\nabla_\mu u, \nu_1]|_{\d\Omega} = 2^{-1/2}( \frac 12 x_1 u_{x_1} + x_2 u_{x_2})\qquad\hbox{on } S_1
\end{equation}
\begin{equation}\label{dwa2}
 0 = [A_\nu\nabla_\mu u ,\nu_2]|_{\d\Omega} = 2^{-1/2}( \frac 12 x_3 u_{x_3} + x_2 u_{x_2})\qquad\hbox{on } S_2.
\end{equation}
The condition of orthogonality to the kernel of $\nabla_\mu$ is $\int_{S_i} u \, d\cH^2 =0$.
\end{proposition}
\begin{proof}
Obviously,  $\bar S_1 \cap \bar S_2 = \{(1,0,-1)\}.$
When we take $\mu = \cH^2\llcorner(S_1\cup S_2)$, then we see that an application of the Relaxation Theorem yields

We shall compute  $[A_\mu \nabla_\mu u,\nu]|_{\d\Omega}$, where $\nu$ is normal to $\Omega$. We notice that
$$
\nabla_\mu u = (u_{x_1}, u_{x_2}, 0)\chi_{S_1} + (0, u_{x_2}, u_{x_3})\chi_{S_2}.
$$
Similarly, we see that
$$
\nu_1 = 2^{-1/2}(x_1,x_2, -1),\qquad x_1^2 +x_2^2 =1
$$
on $S_1$ and
$$
\nu_2 = 2^{-1/2}(1,x_2, x_3),\qquad x_3^2 +x_2^2 =1
$$
on $S_2$.

Thus,  $[A_\mu \nabla_\mu u,\nu]|_{\d\Omega}$  leads to the  conclusion that (\ref{dwa1}) and (\ref{dwa2}) hold.
\end{proof}

We may interpret the last computations 
as the lack of dependence of the boundary conditions on $\Omega$.
\begin{corollary}\label{pa-pa}
 Let us suppose $S$ is a $k$-dimensional manifold contained in $\Omega\subset \bR^N$, $k<N$, and the boundary of $\Omega$ is smooth with outer normal $\nu$. Moreover, $A_\mu$ is a relaxation of matrix $A$. Then,
 $[A_\mu \nabla_\mu u ,\nu ]|_{\d\Omega}=0$ on $\partial\Omega$ is equivalent to
 $[A_\mu \nabla_\mu u , \bn]|_{\d S} =0$ on $\partial\Omega\cap\partial S$, where $\bn$ is the normalized projection of $\nu$ to the normal bundle of $S$.
\end{corollary}
\begin{proof}
Since $\bar S$ intersects $\partial\Omega$ transversally, then $\nu$ is not tangential to
$S$, i.e. $\nu = \alpha \bn + \beta \bn^\perp$, $\alpha\neq 0$. Since $A_\mu$ maps $\bR^N$ into the tangent space of $S$, then we can see that,
$$
[A_\mu \nabla_\mu u, \nu]|_{\d\Omega} = [A_\mu \nabla_\mu u,\alpha \bn + \beta \bn^\perp]|_{\d\Omega}
= [\alpha A_\mu \nabla_\mu u,  \bn]|_{\d S}.
$$
The last equality holds because  $\bn^\perp$ is perpendicular to $T_x S$ at $x\in \partial S$.
\end{proof}

\subsection*{Acknowledgement}
The work of the first author was in part supported by the National Science Centre, Poland, through the grant number
2017/26/M/ST1/00700.



\end{document}